\documentclass[submission,copyright,creativecommons]{eptcs}
\providecommand{\event}{AiML 2026} 

\usepackage{aiml26}

\usepackage{iftex}
\usepackage{newtxtext,newtxmath}

\ifpdf
  \usepackage{underscore}         % Only needed if you use pdflatex.
  \usepackage[T1]{fontenc}        % Recommended with pdflatex
\else
  \usepackage{breakurl}           % Not needed if you use pdflatex only.
\fi

\usepackage{amsmath}
\usepackage{apxproof}
\usepackage[dvipsnames]{xcolor}
\usepackage{stmaryrd}
\usepackage{tikz-cd}
  \usetikzlibrary{arrows,arrows.meta}
  \tikzcdset{arrow style=tikz, diagrams={>=stealth'}}

\usepackage{mathrsfs}
\usepackage[cal=txupr,scr=boondoxupr,scrscaled=1.050]{mathalfa}

\newcommand{\mc}[1]{\mathcal{#1}}
\newcommand{\mb}[1]{\mathbb{#1}}
\newcommand{\mf}[1]{\mathfrak{#1}}

\newcommand{\mo}{\mathfrak}
\newcommand{\cat}{\mathtt}

\renewcommand{\iff}{\quad\text{iff}\quad}
\renewcommand{\phi}{\varphi}
\DeclareMathOperator{\Prop}{Prop}
\newcommand{\llb}{\llbracket}
\newcommand{\rrb}{\rrbracket}
\newcommand{\cl}[1]{\overline{#1}}
\DeclareMathOperator{\interior}{int}
\DeclareMathOperator{\Baire}{Baire}
\newcommand{\tbigvee}{{\textstyle\bigvee}}
\newcommand{\tbigwedge}{{\textstyle\bigwedge}}
\newcommand{\tbigcup}{{\textstyle\bigcup}}

\newcommand{\MML}{\mathsf{K_{\sigma}}}
\newcommand{\MLS}{\mathsf{MLS}}
\newcommand{\lan}{\mathcal{L_{\sigma}}}
\newcommand{\ts}[2]{\llbracket #1 \rrbracket^{\mathfrak{#2}}}
\newcommand{\mts}[2]{\llparenthesis #1 \rrparenthesis^{\mathfrak{#2}}}

\DeclareMathOperator{\clp}{Clp}
\DeclareMathOperator{\uf}{Uf}
\newcommand{\symdif}{\mathbin{\vartriangle}}

\newcommand{\sBA}{\cat{\sigma BA}}
\newcommand{\sStone}{\cat{\sigma Stone}}

\DeclareMathOperator{\op}{op}
\newcommand{\Ax}{\mathrm{Ax}}

\newcommand{\ddiamond}{\Diamonddot}

\newtheorem{problem}[theorem]{Problem}

\title{Modal Measurable Logics via a Modal Loomis-Sikorski Representation Theorem}
\author{%
  Nick Bezhanishvili
  \institute{University of Amsterdam\\ Amsterdam, The Netherlands}
  \email{n.bezhanishvili@uva.nl}
\and
  Jim de Groot\footnote{The second author was supported by Swiss National Science Foundation (SNSF) grant No.~200021\_215157.}
  \institute{University of Bern\\ Bern, Switzerland}
  \email{jim.degroot@unibe.ch}
\and
  Lawrence S.~Moss
  \institute{Indiana University\\ Bloomington, IN, USA}
  \email{lmoss@iu.edu}
}

\newcommand{\titlerunning}{Modal Measurable Logics}
\newcommand{\authorrunning}{N.~Bezhanishvili, J.~de Groot \& L.~Moss}

\hypersetup{
  bookmarksnumbered,
  pdftitle    = {\titlerunning},
  pdfauthor   = {\authorrunning},
  pdfsubject  = {EPTCS},               % Consider adding a more appropriate subject or description
  pdfkeywords = {keyword1, keyword2} % Uncomment and enter keywords specific to your paper
}

\begin{document}
\maketitle

\begin{abstract}
We investigate a modal extension of the infinitary classical logic with countable meets and joins, formulated with
an eye toward measure-theoretic work in dynamical systems and in point-free ergodic theory. We define a modal formalism in this language, which we call modal measurable logics. We also introduce a Kripke-like semantics for these logics in measurable spaces taking a designated modal $\sigma$-ideal into consideration. Using a restriction of J{\'o}nsson-Tarski duality and a modal extension of the Loomis-Sikorski theorem, we prove completeness of modal measurable logics with respect to this new semantics.
\end{abstract}

%%%%%%%%%%%%%%%%%%%%%%%%%%%%%%%%%%%%%%%%%%%%%%%%%%%%%%%%%%%%%%%%%%%%%%%%%%%%%%%%
\section{Introduction}

In (modal) logic, representation theorems are often used to bridge the gap between algebraic and geometric (frame-based) semantics. This allows one to transfer algebraic completeness of a logic to a completeness theorem for its frames.
For instance, the celebrated Stone representation theorem for Boolean algebras yields completeness of classical propositional logic~\cite{CZ97}.
Similarly, Priestley's representation theorem of distributive lattices yields completeness of positive logic with respect to its poset semantics~\cite{DP02},
and Esakia's representation of Heyting algebras~\cite{Esakia74, Esakia19} yields completeness of intuitionistic logic with respect to intuitionistic Kripke frames \cite{CZ97}.
Representation theorems for modal logics are often based on an existing representation for the underlying propositional logic,
such as the J\'onsson–Tarski representation of modal algebras which gives rise to completeness of modal logic with respect to Kripke semantics \cite{BdRV01, CZ97}.
In this paper we carry on with this tradition.

The aim of this paper is to develop a modal formalism for reasoning about relational structures based on measurable spaces. Alternatively, such structures can be viewed as Kripke frames equipped with a measurable space structure.
Specifically, we work with a modal language with countable meets and joins. Ideally, these should be interpreted as intersections and unions; we investigate to what extent this is feasible when aiming for a completeness result.

As a starting point, we consider the countably infinite analogue $\MML$ of the classical modal logic $\mathsf{K}$, obtained by extending the classical propositional logic with countable joins and meets with a unary modal operator $\Diamond$ that distributes over countable joins. We call the algebras corresponding to $\MML$ \emph{$\sigma$-modal algebras}.
These arise from extending $\sigma$-Boolean algebras with a countable-join-preserving operator.
Our first contribution is to prove a J\'onsson–Tarski representation theorem for $\sigma$-modal algebras, which yields order-topological completeness for $\sigma$-modal logics.

However, in the order-topological semantics of $\sigma$-modal logics, countable meets and joins are interpreted as the \emph{interiors} of intersections and the 
\emph{closures} of unions, respectively. While interesting, these structures do not form concrete $\sigma$-algebras in which countable meets and joins are interpreted as set-theoretic intersections and unions themselves.

To remedy this, we prove a modal version of the celebrated Loomis-Sikorski theorem, which was proven (independently) by Loomis~\cite{Loomis47} and Sikorski~\cite{Sikorski48} in the 1940s.
Recall that this theorem states that every ``abstract'' $\sigma$-Boolean algebra (that is, every Boolean algebra in which countable joins and meets exist) is a quotient of a ``concrete'' $\sigma$-algebra of sets 
(where the meets and joins are intersections and unions)
modulo a $\sigma$-ideal. This provides a concrete representation of $\sigma$-Boolean algebras.
Our modal Loomis-Sikorski theorem is one of the main technical contributions of the paper.

It turns out that to prove the modal Loomis-Sikorski theorem, 
we need to extend the logic with a special infinitary rule, which we call the \emph{Infinite Descending Chain} rule (IDC).
Similar rules are not uncommon in the realm of infinitary modal logics, and one
can for instance be found in~\cite{FurEA17}.
We call our logic \emph{Modal Loomis-Sikorski Logic} $\mathsf{MLS}$, and refer to the corresponding algebras as 
\emph{measurable $\sigma$-modal algebras}. We prove:

\begin{quote}
\textbf{Modal Loomis-Sikorski theorem:}
  Every measurable $\sigma$-modal algebra is the quotient of a ``concrete''  $\sigma$-modal algebra by a modal $\sigma$-ideal.
\end{quote}
Based on this result, we develop a ``measurable'' semantics for Modal Loomis-Sikorski Logic based on \emph{modal measurable spaces} with a modal $\sigma$-ideal $N$ of marked sets that we think of as ``designated null-sets''.
A modal measurable space consists of a measurable space $(X, \Sigma)$ and a relation $R$ on $X$ such that $\ddiamond$ (the modal operator corresponding to $R$) maps measurable sets to measurable sets.
Modal formulas are then interpreted in the modal algebra $(\Sigma, \ddiamond)$ quotiented by $N$.

Quotienting by $N$ seems unavoidable, both for technical reasons (in view of the Loomis–Sikorski theorem) and for conceptual ones. The ideal $N$ will typically consist of the null-sets of some probability measure, which is also the perspective adopted by Scott \cite{Scott09}. We note that a similar move was advocated by Lando \cite{Lando12Completeness, Lando12Dynamic, Lando17}, Fern{\'a}ndez-Duque \cite{FernandezDuque10} and 
Pavlov~\cite{Pavlov2020}. Indeed, \cite{Pavlov2020} states
that
``One can argue that an object of the right category of spaces in measure theory is not a set equipped with a $\sigma$-algebra of measurable sets, but rather a set $S$
 equipped with a $\sigma$-algebra $M$
 of measurable sets and a $\sigma$-ideal $N$
 of negligible sets, i.e., sets of measure 0.''
 Moreover, Jamneshan and Tao write in~\cite{JTao2023Foundational} 
that the quotient algebra should be viewed
``as a `point-free' or `pointless' abstraction of [a measurable space] $X$ in which the null-sets have been ‘deleted’.''
They further observe:
``This `point-free' approach to ergodic theory seems particularly well suited for studying actions of uncountable (discrete) groups, as by deleting the null-sets in advance, one can avoid to a large extent the standard difficulty that an uncountable union of null-sets is null.''
A similar approach  appears in G.~Bezhanishvili and Fern{\'a}ndez-Duque \cite{BF24}.

Using our modal version of the Loomis–Sikorski theorem, we prove the main completeness result of this paper:

\begin{quote}
\textbf{Completeness theorem:}
Every axiomatic extension of Modal Loomis-Sikorski Logic is sound and complete with respect to a class of marked modal measurable spaces.
\end{quote}
With our work, we hope to lay the foundation for a modal logic approach to measure-based dynamical systems  and (point-free) ergodic theory.
As a first attempt, we give an example of a modal measurable space used in the analysis of measure-preserving transformations in dynamical systems.

\paragraph{Related work}
One line of work connects modal logic in its probabilistic variant to reasoning about uncertainty in games. This includes early work on the formalisation of Harsanyi type spaces due to Heifetz and Samet~\cite{HeifetzSamet98} (where measure theory first appears
in a prominent way), and later papers which use  
a probabilistic modal logic, such as Moss and Viglizzo~\cite{MossViglizzo06}.
Neither of these papers take up questions of completeness (as we do).
Some papers which do obtain completeness 
include Goldblatt~\cite{Goldblatt10}
and Zhou~\cite{Zhou14}.  What all of the above papers have in common is that their modalities are quantitative: the sentences are of the form $B_p \phi$, where $0 \leq p \leq 1$, which can be read as saying that the probability of (or ``belief'' in) the sentence $\phi$ is at least $p$.
In the current paper, we investigate the use of modal operators to study a more dynamic intuition, rather than a quantitative one.

Another line of loosely related work concerns the combination of modal logic and topology, such as~\cite{BF24, Lando17,Scott09}. This type of work differs from ours because it often uses a topological interpretation of the modalities, following a long line of work on topological interpretations of modal logic, while we do not use the measurable space structure in this way. 

Yet another set of papers 
of potential relevance here are ones which call on dualities related to measure theory, such as~\cite{FurEA17,KozenLMP13}. 
The main difference with our work is the semantic formalism:
in~\cite{FurEA17,KozenLMP13} models are given by
\emph{Markov processes}. These are coalgebras of the probability measure functor, and they might be thought of
as probabilistic analogues of Kripke frames.   That is, they are structures where the accessibility relation comes with
probabilistic information
and the state set is ``mainly just a set.'' 
Our setting, on the other hand, considers a measurable space with an accessibility relation
that is ``mainly just a relation.''
(What we are saying here is not strictly correct: there are connections
between the set-theoretic and measure-theoretic aspects, both our paper and other work.) 
Now in moving closer to our ultimate goal, one would like both a measure-theoretic space and a probabilistic
transition relation.   We view the work in this paper as a necessary step towards this goal.

%%%%%%%%%%%%%%%%%%%%%%%%%%%%%%%%%%%%%%%%%%%%%%%%%%%%%%%%%%%%%%%%%%%%%%%%%%%%%%%%
\section{Stone duality and J\'onsson-Tarski duality}\label{sec:prelim}

  For the results in this section we refer to e.g.~\cite{BdRV01, CZ97, GivHal08}. A \emph{Stone space} is a compact Hausdorff space $(X, \tau)$
  whose clopen subsets form a basis for the topology~$\tau$. Write $\cat{Stone}$ for
  the category of Stone spaces and continuous functions,
  and $\cat{BA}$ for the category of Boolean algebras and (Boolean) homomorphisms.
  
  The collection $\clp(X, \tau)$ of clopen subsets of a Stone space forms
  a Boolean algebra. Moreover, if $f : (X, \tau) \to (Y, \sigma)$ is a
  continuous function between Stone spaces, then
  $f^{-1} : \clp(Y, \sigma) \to \clp(X, \tau)$ is a Boolean homomorphism.
  This gives rise to a contravariant functor
  \begin{equation*}
    \clp : \cat{Stone} \to \cat{BA}.
  \end{equation*}
  Conversely, every Boolean algebra $B$ gives rise to a Stone space
  $\uf(B) := (X_B, \tau_B)$ where $X_B$ is the set of ultrafilters of
  $B$ and $\tau_B$ is the topology on $X_B$ generated by the sets of the form $\theta(b) := \{ u \in X_B \mid b \in u \}$, where $b$ ranges over $B$.
  Furthermore, if $h : A \to B$ is a Boolean homomorphism,
  then $h^{-1} : (X_B, \tau_B) \to (X_A, \tau_A)$ is a continuous
  function. This yields a contravariant functor
  \begin{equation*}
    \uf : \cat{BA} \to \cat{Stone}.
  \end{equation*}

\begin{theorem}[Stone duality]
  $\clp$ and $\uf$ establish a dual equivalence
  $\cat{Stone} \equiv^{\op} \cat{BA}$.
\end{theorem}

  A \emph{$\sigma$-Boolean algebra} is a Boolean algebra with countable meets
  and countable joins, and a $\sigma$-homomorphism is a homomorphism that
  preserves them. Write $\sBA$ for the category of $\sigma$-Boolean algebras
  and $\sigma$-homomorphisms.
  A Boolean algebra $B$ is a $\sigma$-Boolean algebra if and only if its
  dual Stone space $\uf(B)$ satisfies:
  \begin{equation}\label{eq:sStone}
    \text{if } \{ a_n \mid n \in \omega \}
    \text{ is a countable set of clopen sets, then }
    \cl{\tbigcup_{n \in \omega} a_n} \text{ is clopen.}
  \end{equation}
  (Here, if $Y$ is a subset of a topological space,
  then $\cl{Y}$ denotes its closure.)
  In Stone spaces where~\eqref{eq:sStone} holds, we can endow the set
  of clopens with a $\sigma$-BA structure: for a clopen $a$,
we take $- a$ to be the set-theoretic complement of $a$, and 
for
 each countable sequence $(a_n)_{n\in \omega}$ of clopens we define
 \begin{equation}\label{structure-clopens}
    \bigvee_{n\in \omega} a_n = \overline{\bigcup_{n\in \omega} a_n} \qquad \bigwedge_{n\in \omega} a_n = \interior\Big(\bigcap_{n \in \omega} a_n\Big) 
  \end{equation}
  
  Define a $\sigma$-continuous function between $\sigma$-Stone spaces to
  be a function whose dual is a $\sigma$-homomorphism, and write
  $\sStone$ for the category of $\sigma$-Stone spaces and $\sigma$-continuous functions. 
  Then~\cite[Section~39]{GivHal08}:

\begin{theorem}\label{thm:dual-sba}
  Stone duality restricts to a dual equivalence $\sStone \equiv^{\op} \sBA$.
\end{theorem}

Stone duality for Boolean algebras can be extended to a duality for
  modal algebras. By a \emph{modal algebra} we mean a pair $(B, \Diamond)$
  consisting of a Boolean algebra $B$ and a unary operator $\Diamond$
  on $B$ that preserves finite joins.  (Here and always we include the empty join,
  with $\bigvee \emptyset = \bot$.) With the expected notion of homomorphism,
  these form the category $\cat{MA}$ of modal algebras.

  Dually, we work with \emph{modal spaces}.
  These are structures of the form $(X, \tau, R)$ such that $(X, \tau)$
  is a Stone space, $R$ is a point-closed relation (i.e.~$R[x]$ is closed in $(X, \tau)$ for each $x \in X$), and
  $\ddiamond a := \{ x \in X \mid R[x] \cap a \neq \emptyset \}$
  is clopen whenever $a$ is clopen.
  When the topological structure is viewed as a collection of
  admissible sets, modal spaces are often referred to as 
  \emph{descriptive (general) frames}~\cite[Section~5.5]{BdRV01}.
  
  With bounded continuous morphisms, modal spaces comprise the category $\cat{MSpace}$.
  The functor $\clp$ can be extended to a functor
  \begin{equation*}
    \clp : \cat{MSpace} \to \cat{MA}
  \end{equation*}
  by sending a modal space $(X, \tau, R)$ to the modal algebra $(\clp(X, \tau), \ddiamond)$.
  Conversely, if $(B, \Diamond)$ is a modal algebra then
  $(X_B, \tau_B, R_{\Diamond})$ is a modal space, where
  $u R_{\Diamond} v$ iff ($b \in v$ implies $\Diamond b \in u$ for all $b \in B$).
  With the same action on morphisms, the functor $\uf$ extends to
  \begin{equation*}
    \uf : \cat{MA} \to \cat{MSpace}.
  \end{equation*}

\begin{theorem}[J{\'o}nsson-Tarski duality]
  $\clp$ and $\uf$ establish a dual equivalence $\cat{MSpace} \equiv^{\op} \cat{MA}$.
\end{theorem}

  In particular, we note that the map $\theta_B : (B, \Diamond) \to (\clp(X_B, \tau_B), \ddiamond)$, given by $\theta_B(b) := \{ u \in X_B \mid b \in u \}$, is an isomorphism for every modal algebra $(B, \Diamond)$.

%%%%%%%%%%%%%%%%%%%%%%%%%%%%%%%%%%%%%%%%%%%%%%%%%%%%%%%%%%%%%%%%%%%%%%%%%%%%%%%%
\section{The $\sigma$-modal logic and the modal Loomis-Sikorski logic}

  The obvious ``$\sigma$-analogue'' of normal modal logic is the extension
  with countable meets and joins and a diamond that distributes over
  countable joins. This yields what we call \emph{$\sigma$-modal logic}.
  However, in order to prove a modal analogue of the Loomis-Sikorski theorem,
  and obtain completeness with respect to modal measurable spaces modulo
  an ideal of null-sets, we need to stipulate an additional rule.
  Adding this to $\sigma$-modal logic yields what we call
  \emph{Modal Loomis-Sikorski logic}.

\begin{definition}
  The \emph{modal measurable language} is  the collection
  $\lan$ of formulas defined by the grammar
  \begin{equation*}
    \phi ::= p 
        \mid \neg\phi
        \mid \bigvee_{i \in \omega} \phi_i
        \mid \Diamond \phi
  \end{equation*}
  where $p$ ranges over a countably infinite set $\Prop$ of proposition letters.
  We abbreviate $\phi \vee \psi := \bigvee(\phi, \psi, \psi, \psi, \ldots)$
  and $\bot := p \vee \neg p$, and define $\top$, $\to$, $\bigwedge$, $\wedge$
  and $\Box$ as expected.
\end{definition}

\begin{definition}\label{def:logic}
Let $\Ax \subseteq \lan$ be a set of axioms.
The \emph{$\sigma$-modal logic} $\MML \oplus \Ax$ is the least set of $\lan$-formulas derivable from the following axioms and rules: 
\begin{itemize}
  \item All substitution instances of classical tautologies (allowing for countable conjunction and disjunction) and all substitution instances of formulas in $\Ax$
  (this includes the infinite De Morgan laws);
  \item The join rule: for any countable sequence $(\phi_n)_{n \in \omega}$ and formula $\psi$:
        \begin{equation*}
          \dfrac{\phi_n \to \psi \qquad n \in \omega}
                {(\tbigvee_{n \in \omega} \phi_n) \to \psi}
        \end{equation*}
  \item The modal axiom and rule:
        \begin{equation*}
           \Diamond(\tbigvee_{n \in \omega} \phi_n) \leftrightarrow \tbigvee_{n \in \omega}\Diamond\phi_n 
           \qquad
          \dfrac{\neg\phi}{\neg\Diamond\phi}
        \end{equation*}
  \item Modus ponens:
        \begin{equation*}
          \dfrac{\phi \qquad \phi \to \psi}{\psi}
        \end{equation*}
\end{itemize}
\emph{Modal Loomis-Sikorski logic} $\MLS \oplus \Ax$ is the set of $\lan$-formulas
derivable from the axioms and rules above together with the
\emph{infinite descending chain rule (IDC)}:
\begin{equation}\label{eq:inf-diamond} \tag{IDC}
  \dfrac{\neg\tbigwedge_{n \in \omega} \phi_n
         \qquad \phi_{n+1} \to \phi_n \text{ for all } n \in \omega}
        {\neg\tbigwedge_{n \in \omega} \Diamond\phi_n}.
\end{equation}
A formula $\phi$ is called \emph{derivable} in $\MML$ or $\MLS$ if it can be obtained from the axioms and rules of the respective logics.
If $\phi$ is derivable in $\MML$ or $\MLS$ then we write
$\vdash_{\MML} \phi$ or $\vdash_{\MLS} \phi$, respectively.
\end{definition}

  In either logic, we get $\phi_k \to \bigvee_{n \in \omega} \phi_n$
  as a substitution instance of a classical tautology, as well as the following dual axioms and rules:
  \begin{equation*}
    (\tbigwedge_{n \in \omega} \phi_n) \to \phi_k \quad (k \in \omega),
    \qquad
    \Box(\tbigwedge_{n \in \omega} \phi_n) \leftrightarrow \tbigwedge_{n \in \omega} \Box\phi_n,
    \qquad
    \dfrac{\psi \to \phi_n \quad n \in \omega}
          {\psi \to (\tbigwedge_{n \in \omega}\phi_n)},
    \qquad
    \dfrac{\phi}{\Box\phi}.
  \end{equation*}
  Although the use of diamond as a primitive connective provided intuition for some of the proofs,
  using boxes as primitives may yield more natural axioms and rules.

%%%%%%%%%%%%%%%%%%%%%%%%%%%%%%%%%%%%%%%%%%%%%%%%%%%%%%%%%%%%%%%%%%%%%%%%%%%%%%%%
\section{Algebraic semantics for modal $\sigma$-logics}\label{sec:alg}

  Next, we present the algebraic semantics of $\MML$ and $\MLS$.
  Recall that \emph{a $\sigma$-Boolean algebra} is a Boolean algebra that has
  countable meets and countable joins. We extend this to include
  an additional operator,~$\Diamond$.

\begin{definition}
  A \emph{$\sigma$-modal algebra} is a pair $(B, \Diamond)$ consisting of a
  $\sigma$-Boolean algebra $B$ and a unary operator $\Diamond : B \to B$
  that satisfies
  \begin{equation}\label{eq:sMA}
    \Diamond\bot = \bot
    \quad\text{and}\quad
    \bigvee_{n \in \omega} \Diamond b_n = \Diamond \bigvee_{n \in \omega} b_n
  \end{equation}
  for any countable sequence $(b_n)_{n \in \omega}$ of elements in $B$.

  A \emph{measurable $\sigma$-modal algebra} is a
  $\sigma$-modal algebra $(B, \Diamond)$ that additionally satisfies
   the following condition:  
for
  every countable descending chain $(b_n)_{n \in \omega}$
  (i.e.~every countable collection of elements in $B$ such that
  $b_0 \geq b_1 \geq b_2 \geq \cdots$),
  \begin{equation}\label{eq:good}
    \bigwedge_{n \in \omega} b_n = \bot
    \quad\text{implies}\quad
    \bigwedge_{n \in \omega} \Diamond b_n = \bot.
  \end{equation}
  This rule is motivated by the desire for a modal Loomis-Sikorski theorem:
  it dually corresponds to $\ddiamond$ sending Baire null-sets to
  Baire null-sets (Theorem~\ref{thm:JT-restriction}), which is a property used in Lemma~\ref{lemma-immediate} en route to proving the
  Loomis-Sikorski theorem (Theorem~\ref{thm:modal-LS}).
  
  A formula $\phi$ is \emph{valid} in a $\sigma$-modal algebra $(B, \Diamond)$,
  notation: $(B, \Diamond) \Vdash \phi$,
  if it evaluates to $\top$ under every assignment of elements in $B$
  to the proposition
  letters.
  We write $\mathsf{\sigma MA} \Vdash \phi$ if $\phi$ is valid in every $\sigma$-modal algebra,
  and $\mathsf{MMA} \Vdash \phi$ if $\phi$ is valid in every measurable $\sigma$-modal algebra.
  A set $\Ax \subseteq \lan$ of axioms is valid in a $\sigma$-modal algebra
  $(B, \Diamond)$ if all formulas in $\Ax$ are valid.
\end{definition}

\begin{remark}
The condition in~\eqref{eq:good} which we require in
a measurable $\sigma$-modal algebra may be reformulated
to emphasize a connection to compactness.
For all sequences $(b_n)$, not necessarily descending,
\[
 \bigwedge_{\mbox{\scriptsize finite }F \subseteq \omega} \Diamond  \bigwedge_{n\in F} b_n > \bot
    \quad\text{implies}\quad 
    \bigwedge_{n \in \omega} b_n > \bot
\]
It is also possible to reformulate the (IDC) rule in the
logic in a parallel way.
\end{remark}

  Recall that a \emph{$\sigma$-ideal} of a $\sigma$-Boolean algebra $B$
  is a nonempty set closed downward in the order and also under countable joins.
  We call two elements $a, b \in B$ \emph{equivalent modulo $I$},
  and write $a \equiv_I b$,
  if their symmetric difference $a \symdif b := (a \wedge \neg b) \vee (b \wedge \neg a)$ is in $I$, and denote the equivalence class of $b$ by $[b]$.
  Then we can equip $B/I = \{ [b] \mid b \in B \}$ with a $\sigma$-Boolean
  algebra structure by defining $\bot^I = [\bot]$, $\neg^I[b] = [\neg b]$
  and $\bigvee_{n \in \omega}^I [b_n] = [\bigvee_{n \in \omega} b_n]$.
  This extends to $\sigma$-modal algebras as follows.

\begin{definition}
  Let $(B, \Diamond)$ be a $\sigma$-modal algebra.
  A \emph{modal $\sigma$-ideal} is a $\sigma$-ideal $I$ of $B$ such that
  $b \in I$ implies $\Diamond b \in I$ for all $b \in B$.
  We define the \emph{quotient} of $(B, \Diamond)$ by $I$ by
  \begin{equation*}
    (B, \Diamond)/I := (B/I, \Diamond^I)
  \end{equation*}
  where $\Diamond^I[b] = [\Diamond b]$.
\end{definition}

  The next lemma allows us to show that the definition of $\Diamond^I[b]$
  is well defined.

\begin{lemma}\label{lem:msa-quotient}
  Let $(B, \Diamond)$ be a $\sigma$-modal algebra and $I \subseteq B$
  a modal $\sigma$-ideal. Then $a \equiv_I b$ implies $\Diamond a \equiv_I \Diamond b$ for all $a, b \in B$.
\end{lemma}
\begin{proof}
  First, we note
  \begin{equation*}
    \Diamond a
      \leq \Diamond(a \vee b)
      = \Diamond((a \vee b) \wedge (\neg b \vee b))
      = \Diamond((a \wedge \neg b) \vee b)
      = \Diamond(a \wedge \neg b) \vee \Diamond b.
  \end{equation*}
  This implies $\Diamond a \wedge \neg\Diamond b \leq \Diamond(a \wedge \neg b)$.
  Similarly, we can derive $\Diamond b \wedge \neg\Diamond a \leq \Diamond(b \wedge \neg a)$. Therefore, 
  \begin{equation*}
    (\Diamond a) \symdif (\Diamond b)
      = (\Diamond a \wedge \neg\Diamond b) \vee (\Diamond b \wedge \neg\Diamond a)
      \leq \Diamond(a \wedge \neg b) \vee \Diamond(b \wedge \neg a)
      = \Diamond(a \symdif b).
  \end{equation*}
  Now if $a \equiv_I b$ then $a \symdif b \in I$ so by definition
  $\Diamond(a \symdif b) \in I$. Since $I$ is downward closed,
  we find $(\Diamond a) \symdif (\Diamond b) \in I$, as desired.
\end{proof}

\begin{proposition}\label{prop:msa-quotient}
  If $(B, \Diamond)$ is a $\sigma$-modal algebra and $I \subseteq B$ a
  modal $\sigma$-ideal, then $(B, \Diamond)/I$ is a $\sigma$-modal algebra again.
\end{proposition}
\begin{proof}
  We know from Lemma~\ref{lem:msa-quotient} that
  $a \equiv_I b$ implies $\Diamond a \equiv_I \Diamond b$,
  so $\Diamond^I$ is well defined.
  Additionally, it follows immediately from the definitions that
  the equations in~\eqref{eq:sMA} are satisfied, showing that 
  $(B, \Diamond)/I$ is a $\sigma$-modal algebra.
\end{proof}

  Recall that a formula $\phi$ is valid on a $\sigma$-modal algebra $(B, \Diamond)$
  if it evaluates to $\top$ under every evaluation of the proposition letters.
  A rule $\rho$ is valid if for every assignment of the proposition letters:
  if all formulas in the antecedent of $\rho$ evaluate to $\top$, then the consequent also evaluates to $\top$.

\begin{theorem}\label{thm:alg-comp}
  Let $\Ax \subseteq \lan$ be a set of axioms (formulas).
  \begin{enumerate}
    \item $\MML \oplus \Ax$ is sound and complete with respect to the class of
          $\sigma$-modal algebras that validate $\Ax$.
    \item \label{it:alg-comp-mls}
          $\MLS \oplus \Ax$ is sound and complete with respect to the class of
          measurable $\sigma$-modal algebras that validate $\Ax$.
  \end{enumerate}
\end{theorem}
\begin{proof}
  Soundness of all these logics follows from the fact that all the axioms and rules are valid
  in every (measurable) $\sigma$-modal algebra $(B, \Diamond)$.
  For the propositional axioms and rules this follows from the fact
  that $B$ is a $\sigma$-Boolean algebra, and the modal axioms and rules
  correspond precisely to those for the algebras.
For example, it is easy to check that the Infinite Descending Chain rule (IDC) corresponds to 
  Equation~(\ref{eq:good}).

  For completeness, we use a standard Lindenbaum-Tarski construction; this is
  enabled by Proposition~\ref{prop:msa-quotient}.
  Let $L$ be the set of formulas modulo provable equivalence,
  i.e.~$\phi \sim_L \psi$ iff $L \vdash \phi\leftrightarrow \psi$. 
     The usual argument shows that every non-theorem of $L = \MML \oplus \Ax$ is refuted on this algebra. It remains to show that this algebra is a $\MML \oplus \Ax$-algebra, which can be verified using a standard argument.
     
     Finally, for $\MLS \oplus \Ax$ we need to show that the Lindenbaum-Tarski algebra validates $\MLS \oplus \Ax$,  in  particular, that it validates (IDC). However, this also follows from general considerations. Note that 
given a logic $L$, all the rules of $L$ are admissible in it
     \cite{Rybakov97}. Also, in any variety corresponding to a logic $L$, the Lindenbaum-Tarski algebra of $L$ validates a rule $\rho$ if and only if $\rho$ is admissible in $L$ \cite[Theorems~1.3.21 and~1.4.5]{Rybakov97}. Thus, the Lindenbaum-Tarski algebra of $\MLS \oplus \Ax$ validates (IDC) and, as above, it also validates $\Ax$, which finishes the proof. 
\end{proof}

%%%%%%%%%%%%%%%%%%%%%%%%%%%%%%%%%%%%%%%%%%%%%%%%%%%%%%%%%%%%%%%%%%%%%%%%%%%%%%%%
\section{Modal measurable spaces}\label{sec:mms}

We turn to the semantic structures in which we will interpret $\lan$. 
We use measurable spaces with an
  additional relation $R$, where propositional connectives are interpreted in
  the measurable space as usual and the modal operator $\Diamond$ is
  interpreted via $R$. On the one hand, such a semantic structure may be viewed as a modal
  analogue of a measurable space.  On the other hand it is a Kripke frame
  with a measurable space structure.

  We prove that such structures provide a sound semantics for $\MML$,
  and characterise precisely for which such spaces the infinite descending
  chain rule~\eqref{eq:inf-diamond} is valid.
  In~Section~\ref{sec:marked} we take the key step of this paper
  by extending this semantics to incorporate an ideal of
  designated null-sets.  

\begin{definition}\label{def:mod-meas-space}
  A \emph{modal measurable space} is a tuple $\mo{F} = (X, R, \Sigma)$
  consisting of a measurable space $(X, \Sigma)$ and a relation $R$ such that
  for every $a \in \Sigma$ we have
  $\ddiamond a := \{ x \in X \mid R[x] \cap a \neq \emptyset \} \in \Sigma$.
  
  A \emph{modal measurable model} is a pair $\mo{M} = (\mo{F}, V)$
  consisting of a modal measurable space $\mo{F}$ and a valuation
  $V : \Prop \to \Sigma$ that assigns to each proposition letter $p$
  a measurable subset of $X$.
  The interpretation of an $\lan$-formula $\phi$ in $\mo{M}$ is defined
  recursively via
  \begin{equation*}
    \ts{p}{M} = V(p), \qquad
    \ts{\neg\phi}{M} = X \setminus \ts{\phi}{M}, \qquad
    \ts{\tbigvee_{n \in \omega} \phi_n}{M} = \bigcup_{n \in \omega} \ts{\phi_n}{M}, \qquad
    \ts{\Diamond\phi}{M} = \ddiamond \ts{\phi}{M}.
  \end{equation*}
  A formula $\phi$ is \emph{valid} in a modal measurable space
  $\mo{F} = (X, R, \Sigma)$ if $\ts{\phi}{M} = X$ for every modal measurable
  model of the form $\mo{M} = (\mo{F}, V)$.
  We write $\mathsf{MMS} \Vdash \phi$ if $\phi$ is valid on all modal measurable spaces.
\end{definition}

  We clearly have:

\begin{lemma}
  Let $\mo{M}$ be a measurable Kripke model.
  Then $\llb \phi \rrb$ is measurable for each formula $\phi$.
\end{lemma}

  The complex algebra of a modal measurable space is defined as expected:

\begin{definition}
  The \emph{complex algebra} of a modal measurable space $\mo{F} = (X, R, \Sigma)$
  is $\mo{F}^+ := (\Sigma, \ddiamond)$.
\end{definition}

  The semantic interpretation of $\phi$ corresponds to the
  algebraic interpretation in $\mo{F}^+$, viewing an assignment $V$ as an assignment
  for $\mo{F}^+$. The following proposition entails that
  $\MML$ is sound with respect to modal measurable spaces
  by showing that the
  complex algebra of every modal measurable space is a $\sigma$-modal algebra.

\begin{lemma}\label{lem:sound-alg}
  Let $\mo{F} = (X, R, \Sigma)$ be a modal measurable space.
  Then $\mo{F}^+$ is a $\sigma$-modal algebra.
\end{lemma}
\begin{proof}
  We know that $\Sigma$ is a $\sigma$-Boolean algebra, so we only
  verify the equations from~\eqref{eq:sMA}.
  To this end, compute
  $\ddiamond \emptyset = \{ x \in X \mid R[x] \cap \emptyset \neq \emptyset \} = \emptyset$ and
  \begin{equation*}
    \ddiamond \tbigcup_{n \in \omega} b_n
      = \{ x \in X \mid R[x] \cap (\tbigcup_{n \in \omega} b_n) \neq \emptyset \}
      = \{ x \in X \mid R[x] \cap b_n \neq \emptyset \text{ for some } n \in \omega \}
      = \tbigcup_{n \in \omega} \ddiamond b_n.
  \end{equation*}
  It follows that $\mo{F}^+$ is indeed a $\sigma$-modal algebra.
\end{proof}

The following result characterises the modal measurable spaces which validate
  the rule~\eqref{eq:inf-diamond}.   It is not needed in what follows.

\begin{definition} \
  \begin{enumerate}
  \item
  Let $(X, \Sigma)$ be a measurable space.
  We call a set $D \subseteq X$ \emph{measurably compact}
  if for every countable measurable cover $\mc{C} = \{ c_n \mid n \in \omega \}$
  of $X$ there exists a finite subset $\mc{C}' \subseteq \mc{C}$
  that covers $D$.
  \item
  A modal measurable space $(X, R, \Sigma)$ is called
  \emph{pointwise measurably compact} if the set $R[x]$ is measurably compact
  for each $x \in X$.
  \end{enumerate}
\end{definition}

\begin{lemma}\label{lem:ls}
  Let $\mo{F} = (X, R, \Sigma)$ be a modal measurable space.
  Then $\mo{F}^+$ is a measurable $\sigma$-modal algebra if and only
  if $\mo{F}$ is pointwise measurably compact.
\end{lemma}
\begin{proof}
  Suppose that $\mo{F}$ is pointwise measurably compact
  and let $(a_n)_{n \in \omega}$ be a countable sequence of
  elements of $\mo{F}^+$ such that $a_n \geq a_{n+1}$ for all $n \in \omega$,
  and $\bigcap_{n \in \omega} a_n = \emptyset$.
  We need to show that $\bigcap_{n \in \omega} \ddiamond a_n = \emptyset$.

  By assumption $\{ X \setminus a_n \mid n \in \omega \}$ is a measurable cover
  of $X$, so for any $x \in X$ we can find a finite number
  $a_{n_1}, \ldots, a_{n_k}$ of $a_n$'s such that
  $R[x] \subseteq (X \setminus a_{n_1}) \cup \cdots \cup (X \setminus a_{n_k})$.
  It follows that $R[x] \cap a_{n_1} \cap \cdots \cap a_{n_k} = \emptyset$.
  Since the $a_n$ form a descending chain, we can take
  $m := \min \{ n_1, \ldots, n_k \}$ to get $R[x] \cap a_m = \emptyset$.
  This proves $x \notin \ddiamond a_m$, hence $x \notin \bigcap_{n \in \omega} \ddiamond a_n$.
  Since $x \in X$ was arbitrary, we conclude $\bigcap_{n \in \omega} \ddiamond a_n = \emptyset$, as desired.

  For the converse, suppose there exists some countable measurable cover
  $\mc{C}$ of $X$ and a world $x \in X$ such that $R[x]$ cannot be
  covered by any finite subset of $\mc{C}$.
  Write $\mc{C} = \{ c_n \mid n \in \omega \}$.
  Define $a_n := X \setminus (c_0 \cup \cdots \cup c_{n-1})$.
  Then $a_n \supseteq a_{n+1}$ for all $n \in \omega$ and
  $\bigcap_{n \in \omega} a_n = X \setminus \bigcup_{n \in \omega} c_n = X \setminus X = \emptyset$, so the premisses of~\eqref{eq:inf-diamond} are valid.
  But for every $n \in \omega$ we have $R[x] \not\subseteq c_0 \cup \cdots \cup c_{n-1}$,
  so $R[x] \cap a_n \neq \emptyset$, hence $x \in \ddiamond a_n$.
  Therefore $x \in \bigcap_{n \in \omega} \ddiamond a_n \neq \emptyset$,
  showing that~\eqref{eq:inf-diamond} is not valid on $\mo{F}^+$.
\end{proof}

\begin{example}
Here is an example of a 
modal measurable space that is not pointwise measurably compact.
  Consider $\mo{F} = (X, R, \Sigma)$ with
  $X = \omega$, i.e.~the set of natural numbers,
  $R[x] = \omega$ for all $x \in X$ and $\Sigma = \mathcal{P}(\omega)$.
  Then $\mo{F}$ is a modal measurable space, but it is not
  pointwise measurably compact.
  Indeed, $\mathcal{C} = \{ \{ n \} \mid n \in \omega \}$
  is a measurable cover of $X$ but for any $x \in X$,
  since $R[x] = \omega$, there is no finite subcover of $\mathcal{C}$
  covering~$R[x]$.
\end{example}

\begin{theorem}
  The logic $\MML$ is sound with respect to the class of all
  modal measurable spaces.
  The logic $\MLS$ is sound with respect to the class of 
  pointwise measurably compact modal measurable spaces.
\end{theorem}
\begin{proof}
  This follows from Lemmas~\ref{lem:sound-alg}
and~\ref{lem:ls}
  and the fact that the algebraic semantics is sound with respect to the
  two logics.
\end{proof}

We end this section with the following open problem:

\begin{problem}\label{problem}
Are $\MML$ and $\MLS$ complete with respect to suitable classes of modal measurable spaces? If not, can we axiomatise the logics of (pointwise measurably compact) modal measurable spaces?
\end{problem}

%%%%%%%%%%%%%%%%%%%%%%%%%%%%%%%%%%%%%%%%%%%%%%%%%%%%%%%%%%%%%%%%%%%%%%%%%%%%%%%%
\section{Marked modal measurable spaces}\label{sec:marked}

We ended the previous section with the open problem of axiomatising the logic of modal measurable spaces. 
In this section we introduce \emph{marked} modal measurable spaces, which can be considered as ``point-free  analogues'' of modal measurable spaces. 
Marked modal measurable spaces extend modal measurable spaces with a designated $\sigma$-ideal of  (marked) null-sets.
Their use is closely connected to a point-free approach to ergodic theory, see e.g.~\cite{JTao2023Foundational, Pavlov2020}.  We  introduce a new semantics on these spaces and 
prove in Section~\ref{sec:comp} that $\MLS$ is sound and complete with respect to this semantics. 

\begin{definition}
  A \emph{marked modal measurable space} or \emph{mmm-space} is a tuple
  $(X, R, \Sigma, N)$ where
  \begin{itemize}
    \item $(X, R, \Sigma)$ is a modal measurable space; and
    \item $N$ is a modal $\sigma$-ideal on $\mo{F}^+ = (\Sigma, \ddiamond)$
          such that for all countable collections of measurable
          sets $(a_n)_{n \in \omega}$,
          \begin{equation*}
            \bigcap_{n \in \omega} a_n \in N
            \quad\text{implies}\quad
            \bigcap_{n \in \omega} \ddiamond a_n \in N.
          \end{equation*}
  \end{itemize}
  The \emph{complex algebra} of an mmm-model is the $\sigma$-modal algebra
  $\mo{F}^{\dagger} := (\Sigma, \ddiamond)/N$.
\end{definition}

  It follows from Lemma~\ref{lem:sound-alg} and
  Proposition~\ref{prop:msa-quotient}
  that $\mo{F}^{\dagger}$ is a $\sigma$-modal algebra.

\begin{example}\label{bernoulli}
Let $(X,\mathcal{O})$ be a topological space,
let $\Sigma$ be the induced Borel $\sigma$-algebra, 
and let $\mu$ be a probability measure on $\Sigma$.
Let $f : X \to X$ be a \emph{measure-preserving} transformation, as in 
dynamical systems and ergodic theory, i.e.~$f^{-1}(b) \in \Sigma$ and $\mu f^{-1}(b) = \mu(b)$ for all $b \in \Sigma$.
Let $R$ be the graph of $f$, i.e.~$xRy$ iff $f(x) = y$,
and let $N$ be the colletion of null-sets.
Then $\ddiamond a = f^{-1}(a)$ for any measurable set $a$.
Furthermore, for any sequence $(a_n)_{n \in \omega}$ of measurable sets whose intersection is in $N$, 
we have 
$
\bigcap \ddiamond a_n = \bigcap f^{-1}(a_n) = f^{-1}(\bigcap a_n) \in N 
$
because $f$ is measure-preserving.
Therefore $(X, R, \Sigma, N)$ is an mmm-space.

As a concrete example, we may consider $(X, \mathcal{O})$ to be the unit interval with the Euclidean topology, $\mu$ the Lebesgue measure and $f(x)$ the Bernoulli map $f(x) = 2x \mathrel{\mathrm{mod}} 1$.
\end{example}

\begin{definition}
  An \emph{mmm-model} is a pair $\mo{M} = (\mo{F}, V)$ consisting of an
  mmm-space $\mo{F} = (X, R, \Sigma, N)$ and a valuation $V : \Prop \to \Sigma$.
  The \emph{interpretation} of an $\lan$-formula $\phi$ in $\mo{M}$
  is the equivalence class of $\ts{\phi}{}$ in the complex algebra $\mo{F}^{\dagger}$,
  which we denote by $\mts{\phi}{M}$
  for the class of mmm-spaces.

  A formula $\phi$ is said to be \emph{valid} on an mmm-space $\mo{F}$
  if $\mts{\phi}{M}$ is the top element of $\mo{F}^{\dagger}$ for every mmm-model
  $\mo{M} = (\mo{F}, V)$ based on $\mo{F}$.
  In other words, $\phi$ is valid on $\mo{F}$ if $X \setminus \ts{\phi}{M} \in N$
  for every mmm-model $\mo{M}$ based on $\mo{F}$.
For a set $\Ax$,  we write $\cat{MMM}(\Ax)$ for class of mmm-spaces validating all formulas in $\Ax$.
\end{definition}

\begin{theorem}\label{thm:mls-sound}
  For any $\Ax \subseteq \lan$, the logic $\MLS \oplus \Ax$ is sound with
  respect to $\cat{MMM}(\Ax)$.
\end{theorem}
\begin{proof}
  Let $\mo{F} = (X, R, \Sigma, N)$ be an mmm-space that validates $\Ax$.
  Then its complex algebra $\mo{F}^{\dagger}$ is a $\sigma$-modal algebra
  by Proposition~\ref{prop:msa-quotient} and it validates $\Ax$ by
  assumption. Moreover, since $\bigcap_{n \in \omega} a_n \in N$
  implies $\bigcap_{n \in \omega} \ddiamond a_n \in N$, the complex
  algebra also satisfies~\eqref{eq:inf-diamond}.
  Therefore $\mo{F}^{\dagger}$ is a measurable $\sigma$-modal algebra
  that validates $\Ax$. Soundness now follows from the algebraic soundness given in Theorem~\ref{thm:alg-comp}\eqref{it:alg-comp-mls}.
\end{proof}

%%%%%%%%%%%%%%%%%%%%%%%%%%%%%%%%%%%%%%%%%%%%%%%%%%%%%%%%%%%%%%%%%%%%%%%%%%%%%%%%
\section{J{\'o}nsson-Tarski duality for (measurable) $\sigma$-modal algebras}

  We restrict J{\'o}nsson-Tarski duality (see Section~\ref{sec:prelim}) to
  a duality for $\sigma$-modal algebras and measurable $\sigma$-modal algebras.
  In order to prove completeness of $\MLS$ in Section~\ref{sec:comp},
  we only need this restriction on objects, so we focus on restricting the
  object part of J{\'o}nsson-Tarski duality.

\begin{definition}
  Let $(X, \tau)$ be a Stone space.
  A \emph{basic null-set} is a set of the form
  $\bigcap_{n \in \omega} a_n$, where $(a_n)_{n \in \omega}$ is a
  sequence of clopen sets of $(X, \tau)$ such that
  $\interior(\bigcap_{n \in \omega} a_n) = \emptyset$.
  A \emph{Baire null-set} is a subset of $X$ that can be covered by
  countably many basic null-sets.
\end{definition}

\begin{definition}
  A \emph{$\sigma$-modal space} is a modal space $\mb{X} = (X, \tau, R)$ such that
  $(X, \tau)$ is a $\sigma$-Stone space and such that for any countable sequence
  $(a_n)_{n \in \omega}$ of clopen sets,
  \begin{equation}\label{eq:1}
    \ddiamond \Big(\overline{\bigcup_{n \in \omega} a_n}\Big)
      = \overline{\ddiamond \Big(\bigcup_{n \in \omega} a_n \Big)}.
  \end{equation}
  A \emph{measurable $\sigma$-modal space} is a $\sigma$-modal space that additionally
  satisfies that for every Baire null-set $S$, the set
  \begin{equation*}
    \ddiamond S := \{ x \in X \mid R[x] \cap S \neq \emptyset \}
  \end{equation*}
  is a Baire null-set too.
\end{definition}

  We will show that these properties correspond to the
  J{\'o}nsson-Tarski duals being (measurable) $\sigma$-modal algebras.
In the next result, we recall that the $\sigma$-Boolean algebra structure in 
the dual space $\mb{X}$  of a Boolean (or modal) algebra $B$ is given by \eqref{structure-clopens}.

\begin{theorem}\label{thm:JT-restriction}
  Let $(B, \Diamond)$ be a modal algebra and $\mb{X} = (X, \tau, R)$ its
  dual modal space. Then
  \begin{enumerate}
    \item $(B, \Diamond)$ is a $\sigma$-modal algebra if and only if
          $\mb{X}$ is a $\sigma$-modal space;
    \item $(B, \Diamond)$ is a measurable $\sigma$-modal algebra if and only if
          $\mb{X}$ is a measurable $\sigma$-modal space.
  \end{enumerate}
\end{theorem}
\begin{proof}
  (1) \; 
  Suppose $(B, \Diamond)$ is a $\sigma$-modal algebra.
  Then we know from Theorem~\ref{thm:dual-sba} that $(X, \tau)$
  is a $\sigma$-Stone space.
  Furthermore, for any countable sequence $(b_n)_{n \in \omega}$
  of elements in $B$ we have
  \begin{equation*}
    \bigvee_{n \in \omega} \Diamond b_n = \Diamond \bigvee_{n \in \omega} b_n
  \end{equation*}
  so it follows from J{\'o}nsson-Tarski duality that for any sequence
  $(a_n)_{n \in \omega}$ of elements in $\clp(\mb{X})$ we have
  \begin{equation*}
    \cl{\bigcup_{n \in \omega} \ddiamond a_n} = \ddiamond \cl{\bigcup_{n \in \omega} a_n}
  \end{equation*}
  By definition of $\ddiamond$ we have
  $\ddiamond \bigcup_{n \in \omega} a_n = \bigcup_{n \in \omega} \ddiamond a_n$,
  which entails the desired equality
  $\ddiamond(\cl{\bigcup_{n \in \omega} a_n}) = \cl{\ddiamond(\bigcup_{n \in \omega} a_n)}$.
  
  For the converse, assume that $(X, \tau, R)$ satisfies~\eqref{eq:1}.
  By assumption $(X, \tau)$ is a $\sigma$-Stone space and $(B, \Diamond)$
  is a modal algebra, so we know that $B$ is a $\sigma$-Boolean algebra
  and that $\Diamond\bot = \bot$. It remains to show that diamonds distribute
  over countable joins.
  Since $(B, \Diamond) \cong (\clp(\mb{X}), \ddiamond)$ is a modal algebra isomorphism and every such isomorphism is automatically a $\sigma$-modal algebra isomorphism, it suffices to show that the
  dual Boolean algebra of $(X, \tau, R)$ satisfies
  \begin{equation*}
    \cl{\bigcup_{n \in \omega} \ddiamond a_n} = \ddiamond \cl{\bigcup_{n \in \omega} a_n}
  \end{equation*}
  for every sequence $(a_n)_{n \in \omega}$ of elements in $\clp(\mb{X})$.
  In light of~\eqref{eq:1}, it suffices to show
  $\cl{\bigcup_{n \in \omega} \ddiamond a_n} = \cl{\ddiamond \bigcup_{n \in \omega} a_n}$. But this follows from
  the fact that
  $\bigcup_{n \in \omega} \ddiamond a_n = \ddiamond \bigcup_{n \in \omega} a_n$,
  which is an immediate consequence of the definition of $\ddiamond$.
  
  (2) \; 
  Suppose that $(B, \Diamond)$ is a measurable $\sigma$-modal algebra.
  Then using the isomorphism $(B, \Diamond) \cong (\clp(\mb{X}), \ddiamond)$,
  we know that for any countable descending chain
  $a_0 \supseteq a_1 \supseteq a_2 \supseteq \cdots$
  of clopen subsets of $\mb{X}$,
  \begin{equation*}
    \interior\Big(\bigcap_{n \in \omega} a_n\Big) = \emptyset
    \quad\text{implies}\quad
    \interior\Big(\bigcap_{n \in \omega} \ddiamond a_n\Big) = \emptyset.
  \end{equation*}
  Let $S$ be a Baire null-set.
  Then $S \subseteq \bigcup_{n \in \omega} S_n$ for some collection $(S_n)_{n \in \omega}$
  of basic null-sets.
  Since $\ddiamond S \subseteq \ddiamond(\bigcup_{n \in \omega} S_n) = \bigcup_{n \in \omega} \ddiamond S_n$
  and the countable union of Baire null-sets sets is a Baire null-set (by definition),
  it suffices to show that $\ddiamond S_n$ is a Baire null-set for any basic null-set $S_n$.
  
  So let $S = \bigcap_{n \in \omega} a_n$ be a basic null-set, where $(a_n)_{n \in \omega}$
  is a sequence of clopen subsets of $\mb{X}$ such that $\interior(\bigcap_{n \in \omega} a_n) = \emptyset$.
  Defining $a'_n := a_0 \cap \cdots \cap a_n$ yields a countable descending chain
  such that $\bigcap_{n \in \omega} a'_n = \bigcap_{n \in \omega} a_n$,
  hence $\interior(\bigcap_{n \in \omega} a'_n) = \emptyset$.
  This implies
  $\interior\Big( \bigcap_{n \in \omega} \ddiamond a'_n \Big) = \emptyset$.
  If we can prove that
  \begin{equation}\label{eq:3}
    \bigcap_{n \in \omega} \ddiamond a'_n = \ddiamond S
  \end{equation}
  then we are done.
  To this end, let $x \in \ddiamond S$. Then there is some $y$ such that $xRy$ and $y \in S$.
  This implies $y \in a'_n$, hence $x \in \ddiamond a'_n$, for all $n \in \omega$.
  Therefore $x \in \bigcap_{n \in \omega} \ddiamond a'_n$.
  For the converse, suppose $x \in \bigcap_{n \in \omega} \ddiamond a'_n$.
  Then using the fact that the $a'_n$ are decreasing, one can show that
  $\{ R[x] \} \cup \{ a'_n \mid n \in \omega \}$ has the finite
  intersection property.
  So by compactness there exists some $y \in R[x] \cap \bigcap_{n \in \omega} a'_n = R[x] \cap S$,
  so $x \in \ddiamond S$. This proves~\eqref{eq:3}.
  
  Conversely, suppose $\ddiamond S$ is a Baire null-set for any Baire null-set $S$.
  We show that $(B, \Diamond)$ is a measurable $\sigma$-modal algebra by checking (\ref{eq:good}).
  Let $(a_n)_{n \in \omega}$ be a countable descending chain of clopen sets such 
  that $\interior(\bigcap_{n \in \omega} a_n) = \emptyset$.
  Then $\bigcap_{n \in \omega}a_n$ is a basic null-set by definition.
  The assumption now gives that $\ddiamond(\bigcap_{n \in \omega} a_n)$
  is a Baire null-set.
  Using~\eqref{eq:3} we then conclude that $\bigcap_{n \in \omega} \ddiamond a_n$
  is a Baire null-set, as desired.
\end{proof}

\begin{remark}
  It follows immediately from the definitions that every Baire null-set
  in a Stone space is a meager set, but it is unclear whether the converse
  holds. We leave this as an open question.
\end{remark}

%%%%%%%%%%%%%%%%%%%%%%%%%%%%%%%%%%%%%%%%%%%%%%%%%%%%%%%%%%%%%%%%%%%%%%%%%%%%%%%%
\section{Completeness via a modal Loomis-Sikorski representation}\label{sec:comp}

  While Theorem~\ref{thm:JT-restriction} provides a duality for the algebraic
  semantics of both $\MML$ and $\MLS$, it does not yield completeness for
  either with respect to (marked) modal measurable spaces.
  This is because countable meets and joins in (measurable) $\sigma$-modal spaces
  need not be given by intersections and unions.
  In order to obtain a completeness result for $\MLS$ with respect to
  mmm-spaces, we prove a modal analogue of the Loomis-Sikorski theorem
  which states that every measurable modal algebra can be obtained as the
  quotient of a modal measurable space by a modal $\sigma$-ideal.

\begin{definition}
  Let $(X, R)$ be a Kripke frame.
  A \emph{modal $\sigma$-field} on $(X, R)$ is a collection
  $\mc{F} \subseteq \mc{P}X$ that contains $X$ and is closed under
  countable unions, complements and $\ddiamond$
  (hence also contains $\emptyset$ and is closed under countable intersections).
\end{definition}

  We note that a modal measurable space (Definition~\ref{def:mod-meas-space}) can be viewed as a triple $(X, R, \Sigma)$ consisting of a Kripke frame $(X, R)$ together with a modal $\sigma$-field $\Sigma$ on $(X, R)$.

  The arbitrary intersection of modal $\sigma$-fields on a Kripke frame
  is again a modal $\sigma$-field. Therefore we can define the modal
  $\sigma$-field on $(X, R)$ generated by some collection $C$ of subsets of
  $X$ as the intersection of all modal $\sigma$-fields containing $C$.

\begin{definition}
  Let $\mb{X} = (X, \tau, R)$ be a $\sigma$-modal space.
  Then we write $\Baire(\mb{X})$ for the modal $\sigma$-field on $(X, R)$
  generated by the clopen subsets of $(X, \tau)$.
  We call sets in $\Baire(\mb{X})$ \emph{modal Baire sets}.
\end{definition}

  In analogy with the non-modal case (see e.g.~\cite[Chapter~40]{GivHal08})
  we find that the modal Baire null-sets, i.e.~sets of a
  Stone space that are both modal Baire sets and Baire null-sets,
  form a modal $\sigma$-ideal.
  This only works for measurable $\sigma$-modal spaces, because we need
  the fact that $\ddiamond S$ is a Baire null-set whenever $S$ is a Baire null-set.
  As a consequence, the modal analogue of the Loomis-Sikorski theorem
  (Theorem~\ref{thm:modal-LS}) only applies to measurable $\sigma$-modal algebras.

\begin{lemma}\label{lemma-immediate}
  Let $\mb{X} = (X, \tau, R)$ be a measurable $\sigma$-modal space.
  Then the collection $\mc{M}$ of modal Baire null-sets of $(X, \tau)$ forms a
  modal $\sigma$-ideal of $\Baire(\mb{X})$.
\end{lemma}
\begin{proof}
  It follows from the definition of a Baire null-set 
  that $\mc{M}$ is a $\sigma$-ideal,
  and it is closed under $\ddiamond$ because by the definition of a
  measurable $\sigma$-modal space, $\ddiamond S$ is a Baire null-set for any Baire null-set~$S$.
\end{proof}

\begin{lemma}
  Every modal Baire set of a measurable $\sigma$-modal space
  is equivalent to a unique clopen set modulo
  the modal $\sigma$-ideal $\mc{M}$ of modal Baire null-sets.
\end{lemma}
\begin{proof}
  Let $\mb{X} = (X, \tau, R)$ be a measurable $\sigma$-modal space and
  $\mc{M}$ the modal $\sigma$-ideal of modal Baire null-sets.
  Let $\mc{F}$ be the class of subsets of $X$ that are equivalent to a
  clopen subset modulo $\mc{M}$. We claim that $\mc{F}$ is a modal $\sigma$-field on $(X, R)$ that 
  includes $\Baire(\mb{X})$.

  To this end, first we note that every clopen subset $a$ of $\mb{X}$
  is in $\mc{F}$ because $a \symdif a = \emptyset \in \mc{M}$.
  To see that $\mc{F}$ is closed under complementation, let $S \in \mc{F}$
  and suppose $a$ is a clopen set such that $S \equiv_{\mc{M}} a$,
  i.e.~$S \symdif a \in \mc{M}$.
  Then
  \begin{equation*}
    (X \setminus S) \symdif (X \setminus a) = S \symdif a \in \mc{M}
  \end{equation*}
  and since $X \setminus a$ is clopen, we get $X \setminus S \in \mc{F}$.

  Next, let $( S_n )_{n \in \omega}$ be a sequence of sets in $\mc{F}$,
  and let $(a_n)_{n \in \omega}$ be a sequence of clopen sets such that
  $S_n \equiv_{\mc{M}} a_n$ for each $n \in \omega$.
  Then we have
  \begin{equation*}
    \Big(\bigcup_{n \in \omega} S_n\Big) \symdif \Big(\bigcup_{n \in \omega} a_n \Big) \subseteq \bigcup_{n \in \omega}(S_n \symdif a_n) \in \mc{M},
  \end{equation*}
  so $\bigcup_{n \in \omega} S_n \equiv_{\mc{M}} \bigcup_{n \in \omega} a_n$.
  Now let $a := \cl{\bigcup_{n \in \omega} a_n}$. Then $a$ is clopen because $\mb{X}$ is a $\sigma$-modal space. Moreover, we have
  \begin{equation*}
    a \symdif \Big(\bigcup_{n \in \omega} a_n\Big)
      = a \setminus \Big(\bigcup_{n \in \omega} a_n\Big)
      = a \cap \bigcap_{n \in \omega} (X \setminus a_n),
  \end{equation*}
  which is in $\mc{M}$ because
  \begin{equation*}
    \interior{\Big(a \cap \bigcap_{n \in \omega} (X \setminus a_n)\Big)}
      = \interior{(a)} \cap \interior{\Big(\bigcap_{n \in \omega} X \setminus a_n\Big)}
      = \interior{(a)} \cap X \setminus \Big(\cl{\bigcup_{n \in \omega} a_n}\Big)
      = \interior{(a)} \cap (X \setminus a)
      = \emptyset.
  \end{equation*}
  So we have $\bigcup_{n \in \omega} S_n \equiv_{\mc{M}} \bigcup_{n \in \omega} a_n \equiv_{\mc{M}} a$.
  Hence $\bigcup_{n \in \omega} S_n \in \mc{F}$.
  
  Finally, if $S \in \mc{F}$ and $S \equiv_{\mc{M}} a$, where $a$
  is clopen, then Lemma~\ref{lem:msa-quotient} entails $\ddiamond S \equiv_{\mc{M}} \ddiamond a$.
  Since $\ddiamond a$ is clopen, it follows that $\ddiamond S \in \mc{F}$.
  Therefore $\mc{F}$ is a modal $\sigma$-field on $(X, R)$
  that contains the clopen subsets, hence $\Baire(\mb{X}) \subseteq \mc{F}$.
  This shows that every modal Baire set is equivalent modulo $\mc{M}$
  to a clopen subset.

  For uniqueness, suppose that the modal Baire set $S$ is equivalent
  modulo $\mc{M}$ to each of the clopen sets $a$ and $b$.
  Then $a \equiv_{\mc{M}} b$, so $a \symdif b \in \mc{M}$.
  Then $a \symdif b$ is open because $a$ and $b$ are clopen,
  and since every Baire null-set is in particular meager
  we can use~\cite[Theorem~28]{GivHal08}
  (which is a corollary of the Baire Category Theorem)
  to conclude $a \symdif b = \emptyset$.
  Therefore $a = b$.
\end{proof}

  We can now prove a modal analogue of the Loomis-Sikorski
  theorem~\cite[Chapter~40, Corollary~1]{GivHal08}.

\begin{theorem}\label{thm:modal-LS}
  Every measurable $\sigma$-modal algebra is isomorphic to a modal
  $\sigma$-field modulo a modal $\sigma$-ideal.
\end{theorem}
\begin{proof}
  Let $(B, \Diamond)$ be a $\sigma$-modal algebra,
  $\mb{X} = (X, \tau, R)$ its dual $\sigma$-modal space,
  and $\Baire(\mb{X})$ the modal $\sigma$-field of modal Baire sets of $\mb{X}$.
  Let $f$ be the function that takes each $S \in \Baire(\mb{X})$ to the unique
  clopen set $a$ to which $S$ is equivalent.
  Then $f$ is a surjective function. 
  
  We claim that $f$ is a $\sigma$-homomorphism. Clearly $f(X) = X$ and $f(\emptyset) = \emptyset$.
  If $S \in \Baire(\mb{X})$ and $S \equiv_{\mc{M}} a$
  then we know $(X \setminus S) \equiv_{\mc{M}} (X \setminus a)$.
  This implies $f(X \setminus S) = X \setminus a = X \setminus f(S)$,
  so $f$ preserves negations. Similarly, if $(S_n)_{n \in \omega}$
  is a countable sequence of modal Baire sets and $f(S_n) = a_n$ for
  each $n \in \omega$, then we know that
  \begin{equation*}
    f\Big(\bigcup_{n \in \omega} S_n\Big)
      = \cl{\bigcup_{n \in \omega} a_n}
      = \bigvee_{n \in \omega} a_n
      = \bigvee_{n \in \omega} f(S_n).
  \end{equation*}
  Finally, we check that $f$ commutes with $\ddiamond$.
Let $\ddiamond S \in \Baire(\mb{X})$, and let $a$ be the unique clopen set such that $S \equiv_{\mc{M}} a$.
  Then $f(S) = a$.
  Moreover, by Lemma~\ref{lemma-immediate}, 
 we have $f(\ddiamond S) = \ddiamond a$ as well: $S\equiv_{\mc{M}} a$ implies that $\ddiamond S\equiv_{\mc{M}} \ddiamond a$, and this just means that $f(\ddiamond S) = \ddiamond a$.
  So
    $f(\ddiamond S) = \ddiamond a = \ddiamond f(S)$.
All in all, $f$ is a surjective modal $\sigma$-homomorphism.  So
 $(B, \Diamond) \cong \Baire(\mb{X})/\mc{M}$.
\end{proof}

  This sets us up to prove completeness of $\MLS \oplus \Ax$ with respect to
  the class of mmm-spaces validating all axioms in $\Ax$.

\begin{theorem}\label{LS-compl}
  The logic $\MLS \oplus \Ax$ is sound and complete with
  respect to the class of mmm-spaces validating $\Ax$.
  That is, is for every formula $\varphi$ we have 
  \begin{equation*}
    \MLS \oplus \Ax \vdash \phi \iff  \cat{MMM}(\Ax) \Vdash \phi.
  \end{equation*}
\end{theorem}
\begin{proof}
  We proved soundness in Theorem~\ref{thm:mls-sound}.
  For completeness, suppose $\phi$ is a formula such that
  $\MLS \oplus \Ax \not\vdash \phi$. Then by Theorem~\ref{thm:alg-comp}
  there exists a  measurable $\sigma$-modal algebra $(B, \Diamond)$ such that
  $(B, \Diamond) \Vdash \Ax$ and $(B, \Diamond) \not\Vdash \varphi$.
  By the proof of the Modal Loomis-Sikorski Theorem~\ref{thm:modal-LS}
  there exists a modal measurable space $(X, R, \Sigma)$
  and a modal $\sigma$-ideal $\mc{M}$ that is closed under $\ddiamond$
  such that $(\Sigma, \ddiamond)/\mc{M}$ is isomorphic to  $(B, \Diamond)$.
  This implies that $\mf{F} = (X, R, \Sigma, \mc{M})$ is an mmm-space
  such that $\mf{F} \not\Vdash \phi$, as desired. 
\end{proof}

%%%%%%%%%%%%%%%%%%%%%%%%%%%%%%%%%%%%%%%%%%%%%%%%%%%%%%%%%%%%%%%%%%%%%%%%%%%%%%%%
\section{Conclusion}

In this paper we introduced a new semantics and logic for reasoning about relational structures based on measurable spaces.
Our main method is a modal version of the Loomis–Sikorski theorem, which yields a completeness result for these infinitary logics with respect to the new semantics.

We view this work as a first step in developing a modal logical framework for reasoning about measure-based dynamical systems and (point-free) ergodic theory. This naturally gives rise to many avenues for future research, the first of which was already mentioned in Problem~\ref{problem} above.

Second, as mentioned in the introduction, our framework considers the standard Kripke semantics but based on measurable spaces. A natural next step would be to make the relations probabilistic and, more generally, to incorporate probability explicitly into the framework.
Orthogonal to this, it would be interesting to investigate whether the logic and semantics benefit a coalgebraic perspective.

Finally, the framework used in this paper could be generalised to other algebraic structures arising in classical logic, such as distributive lattices and (modal) Heyting algebras. This leads to the question of whether there are Loomis–Sikorski type theorems for distributive lattices, (modal) Heyting algebras, and what the corresponding logics, semantics, and completeness results would be.
Ultimately, we would like to develop constructive measure theory based on $\sigma$-Heyting algebras, with an appropriate analogue of the Loomis-Sikorski theorem.

\bigskip\noindent
{\bf Acknowledgments} The authors would like to thank the anonymous referees for their careful reading and many helpful comments. They are also grateful to Rodrigo Almeida for useful discussions on the completeness of infinitary logics.

%%%%%%%%%%%%%%%%%%%%%%%%%%%%%%%%%%%%%%%%%%%%%%%%%%%%%%%%%%%%%%%%%%%%%%%%%%%%%%%%
\bibliographystyle{eptcs}
\bibliography{biblio}

\end{document}